%% file: GrunbaumConjecture.tex
\documentclass[10pt]{article}
\usepackage{a4wide}
  \usepackage{amsmath}
  \usepackage{amssymb}
  \usepackage{latexsym}
  \usepackage[pdftex]{graphicx}
  \usepackage{color}
\usepackage[mathscr]{euscript}

  \newenvironment{proof}{\vspace{1ex}\noindent{\bf Proof:}}{\hspace*{\fill}$\blacksquare$\vspace{1ex}}
  \newenvironment{proofof}[1]{\vspace{1ex}\noindent{\bf Proof of #1:}}{\hspace*{\fill}$\blacksquare$\vspace{1ex}}

  \newtheorem{theorem}{Theorem} %%[section]
  \newtheorem{lemma} [theorem] {Lemma}%[section]
  \newtheorem{corollary} [theorem] {Corollary}%[section]
\newcommand{\eR}[0]{\ensuremath{ \mathbb R}}

\newcommand{\eN}[0]{\ensuremath{ \mathbb N}}

\newcommand{\Fscr}[0]{\ensuremath{{\mathscr F}}}

\DeclareMathOperator{\conv}{conv}

\title{A counterexample to conjecture 18.5 in ``Geometric Etudes in Combinatorial Mathematics'', second edition}
  \author{Tobias M\"uller\thanks{Utrecht University and Centrum Wiskunde \& Informatica. E-mail: \texttt{tobias@cwi.nl}.  
Supported in part by a VENI grant from Netherlands Organisation for Scientific Research (NWO)}}
  
\begin{document}

  \maketitle

\begin{abstract}
A collection of sets $\Fscr$ has the $(p,q)$-property if out of every 
$p$ elements of $\Fscr$ there are $q$ that have a point in common.
A transversal of a collection of sets $\Fscr$ is a set $A$ that intersects every member of
$\Fscr$.
Gr\"unbaum conjectured that every family $\Fscr$ of closed, convex sets in the plane
with the $(4,3)$-property and at least two elements that are compact
has a transversal of bounded cardinality.
Here we construct a counterexample to his conjecture.
On the positive side, we also show that if such a collection $\Fscr$ contains two {\em disjoint} compacta then there is a
transveral of cardinality at most 13.
\end{abstract}

%%%%%%%%%%%%%%%%%%%%%%%%%%%%%%%%%%%%%%%%%%%%%%%%%%%%%%%%%%%%%%%%%%%%%%%%%%%%%%%
%%%%%%%%%%%%%%%%%%%%%%%%%%%%%%%%%%%%%%%%%%%%%%%%%%%%%%%%%%%%%%%%%%%%%%%%%%%%%%%

\section{Introduction and statement of results}

Let $\Fscr$ be a collection of sets.
A {\em transversal} of $\Fscr$ is a set $A$ that intersects every member of $\Fscr$
(that is, $A \cap F \neq \emptyset$ for all $F \in\Fscr$).
The {\em transversal number} of {\em piercing number} $\tau(\Fscr)$ of $\Fscr$ is
the smallest size of a transversal, i.e.

\[ \tau(\Fscr) := \min_{A \text{ transversal of } \Fscr } |A|. \]

\noindent
(Note that $\tau(\Fscr)=\infty$ if no finite transversal exists.)

A collection of sets $\Fscr$ has the $(p,q)$-property if out of every $p$ sets of $\Fscr$
there are $q$ that have a point in common.
In 1957, Hadwiger and Debrunner~\cite{HadwigerDebrunner1957} conjectured
that for every $d$ and every $p \geq q \geq d+1$ there is a universal constant
$c = c(d;p,q)$ such that every finite collection $\Fscr$ of convex sets in $\eR^d$ with 
the $(p,q)$-property satisfies $\tau(\Fscr) \leq c$.
(By considering hyperplanes in general position it is easily seen that for $q\leq d$ no such
universal constant $c$ can exist.)
Many years later, in 1992, Alon and Kleitman~\cite{AlonKleitman1992} finally proved the
conjecture of Hadwiger and Debrunner by cleverly combining various pre-existing tools from the literature.

In the special case when $p=q=d+1$ the Hadwiger-Debrunner conjecture reduces to the classical theorem of 
Helly~\cite{Helly1923} which states that if $\Fscr$ is a finite family of convex sets in $\eR^d$ such that 
every $d+1$ members of $\Fscr$ have a point in common then $\tau(\Fscr) = 1$.
A variant of Helly's theorem states that if $\Fscr$ is an infinite collection
of closed, convex sets in $\eR^d$ and at least one member of $\Fscr$ is compact then
$\tau(\Fscr) = 1$.

Erd\H{o}s conjectured that in the first nontrivial case of the Hadwiger-Debrunner problem, a
similar variant would be true. That is, he conjectured that if $\Fscr$ is a collection of closed, convex sets
in the plane with the $(4,3)$-property and one of the members of $\Fscr$ is compact, then 
$\tau(\Fscr) \leq c$ for some universal constant $c$. 
Boltyanski and Soifer included this conjecture in the first edition of their book ``Geometric Etudes in Combinatorial Mathematics''
and they offered a prize of \$25 for its solution.
Eighteen years later, Gr\"unbaum found a simple counterexample while proofreading the second edition, earning the reward.
Gr\"unbaum also made a conjecture of his own, stating that if $\Fscr$ is as above, and {\em two} members of $\Fscr$ are compact
then $\tau(\Fscr)$ is finite. (See~\cite{SoiferBoek}, pages 198-199.)
Here we show that Gr\"unbaum's conjecture fails as well:

\begin{theorem}
There exists a collection $\Fscr$ of closed, convex subsets of the plane
such that 
\begin{enumerate} 
\item $\Fscr$ has the $(4,3)$-property, and;
\item Two of the elements of $\Fscr$ are compact, and;
\item $\tau(\Fscr) = \infty$. 
\end{enumerate}
\end{theorem}

\noindent
On the positive side, we show that any collection $\Fscr$ of closed, convex sets in the plane that contains two disjoint compacta
and satifies the $(4,3)$-property
does have universally bounded transversal number:

\begin{theorem}\label{thm:disjoint}
If $\Fscr$ is a collection of closed, convex sets 
in the plane such that 
\begin{enumerate}
\item $\Fscr$ has the $(4,3)$-property, and;
\item $\Fscr$ contains two disjoint compacta, 
\end{enumerate}
then $\tau(\Fscr) \leq 13$. 
\end{theorem}

\section{The counterexample}

Let us set

\[ F_1 := [-1,1]\times\{0\}, \quad F_2 := [0,2] \times \{0\}. \]

\noindent
Let $t_1 < t_2 < t_3 < \dots$ be a strictly increasing sequence of numbers between $0$ and $1$, and let
$s_1 > s_2 > \dots$ be a strictly decreasing sequence of negative numbers
that tends to $-\infty$. 
(For instance $t_n := 1 - \frac{1}{n}, s_n := -n$ would be a valid choice.)
Set $p_n := (t_n, 0)$; let $\ell_n$ denote the vertical line through $p_n$; and
let $\ell_n'$ denote the line through $p_n$ of slope $s_n$.

For $n \geq 3$ we now let $F_n$ be the set of all points either on or to the left of $\ell_n$ and either on or above $\ell_n'$.
See figure~\ref{fig:counterex}.

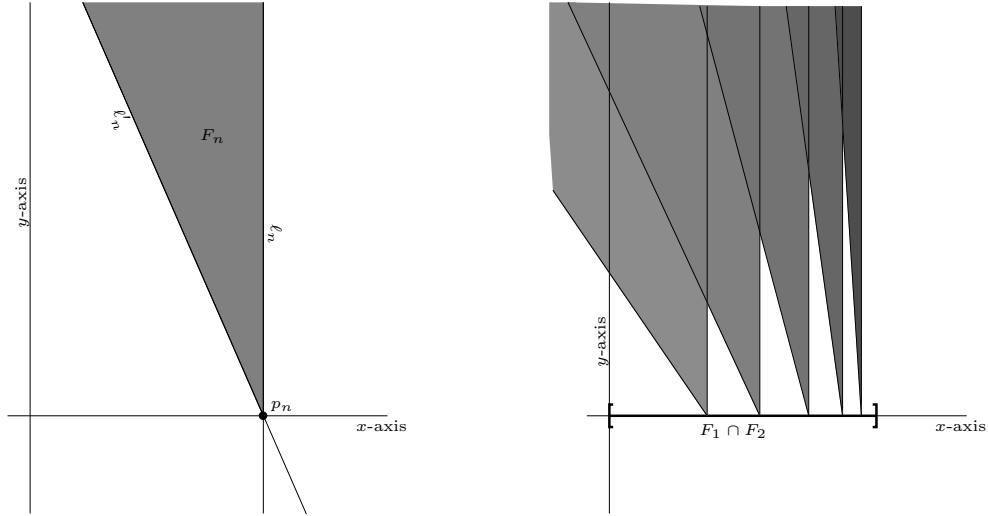
\begin{figure}[h!]
\begin{center}
\input{Fn4.pspdftex} 
\end{center}
\caption{The construction of $F_n$ for $n \geq 3$ (left)
and part of the collection $\Fscr$ (right).\label{fig:counterex}}
\end{figure}

\noindent
Observe that by construction $F_n$ contains all sufficiently high points on the $y$-axis for all $n\geq 3$:

\begin{equation}\label{eq:yn}
\text{For each $n\geq 3$ there exists a $y_n > 0$ such that
$\{ (0,y) : y \geq y_n \} \subseteq F_n$.}
\end{equation}

\noindent
Let $\Fscr := \{ F_1, F_2, \dots \}$ be the resulting infinite collection of
closed convex sets.
We first establish that $\Fscr$ has the $(4,3)$-property.

\begin{lemma}
$\Fscr$ has the $(4,3)$-property. 
\end{lemma}

\begin{proof}
Let us pick four abitrary distinct indices $i_1 < i_2 < i_3 < i_4$
and consider the quadruple $F_{i_1}, F_{i_2}, F_{i_3}, F_{i_4} \in \Fscr$.

If $i_1=1$ and $i_2=2$ then clearly $F_{i_1} \cap F_{i_2} \cap F_{i_3} = \{p_{i_3}\}$, so that
$F_{i_1}, F_{i_2}, F_{i_3}$ is an intersecting triple.
We can thus assume that $i_2, i_3, i_4 > 2$.
In this case $F_{i_2} \cap F_{i_3} \cap F_{i_4} \neq \emptyset$ by the observation~\eqref{eq:yn}.
\end{proof}

\noindent
It remains to show that $\Fscr$ does not have a finite transversal.

\begin{lemma}
$\tau(\Fscr) = \infty$.
\end{lemma}

\begin{proof}
It suffices to show that every point of the plane is in finitely many
elements of $\Fscr$.
Let $a = (a_x,a_y) \in \eR^2$ be arbitrary. 
If $a_y \leq 0$ then $a$ is in at most three elements of $\Fscr$. Let us therefore assume $a_y > 0$.
In this case, if $a_x \geq t_n$ for all $n \in \eN$ then $a$ is in no element of $\Fscr$. Let us therefore assume that 
there is at least one $n\in\eN$ such that $a_x < t_n$.
Let us fix an $n_0$ such that $a_x < t_{n_0}$, and set 

\[ s := - \frac{a_y}{t_{n_0} - a_x}. \]%
(Note that $s$ is exactly the slope of the line through $a$ and $p_{n_0}$.) \\
Since $s_n \to -\infty$, there is an $m_0$ such that $s_n < s$ for all $n \geq m_0$.

Observe that for all $n \geq \max(n_0,m_0)$ the point $a$ is below
the line $\ell_n'$ (as the point $p_n$ is to the right of $p_{n_0}$ and
$\ell_n'$ has a steeper slope than $s$).
This shows that $a \not\in F_n$ for all $n \geq \max(n_0,m_0)$.
Hence $a$ is in finitely many elements of $\Fscr$ as required.
\end{proof}

\noindent
{\bf Remark:} By adding additional compact sets to $\Fscr$ that each contain $[0,1]\times\{0\}$ 
we can obtain a collection $\Fscr'$ that contains 
an arbitrary number of compacta, and still has the $(4,3)$-property and $\tau(\Fscr') = \infty$.

\section{The proof of Theorem~\ref{thm:disjoint}}

The proof of the Hadwiger-Debrunner conjecture by Alon and Kleitman~\cite{AlonKleitman1992}
does not give a good bound on the universal constant $c$.
A better bound on this constant for the special case when $p=4, q=3$ was later
given by Kleitman, Gyarfas and T\'oth~\cite{KleitmanEtal2001}.

\begin{theorem}[Kleitman et al.~\cite{KleitmanEtal2001}]
If $\Fscr$ is a finite collection of convex sets in the plane
with the $(4,3)$-property then $\tau(\Fscr) \leq 13$. 
\end{theorem}

\noindent
A standard compactness argument (which we do not repeat here)
shows that the same also holds if $\Fscr$ is an infinite collection
of convex compacta with the $(4,3)$-property.

\begin{corollary}\label{cor:Kleitman}
If $\Fscr$ is an infinite collection of convex, compact sets in the plane
and $\Fscr$ has the $(4,3)$-property then 
$\tau(\Fscr) \leq 13$. 
\end{corollary}

\begin{proofof}{Theorem~\ref{thm:disjoint}}
Let $\Fscr$ be an arbitrary infinite collection of
closed, convex sets with the $(4,3)$-property with two sets
$A, B \in \Fscr$ that are disjoint and compact.
Let us set 

\[ F_0 := \conv(A\cup B). \]%
%
%Then $F_0$ is clearly compact.
Let us first observe that

\begin{equation}\label{eq:obs1} 
F \cap F_0 \neq \emptyset \text{ for all } F \in\Fscr. 
\end{equation}

\noindent
To see this, suppose that some $F \in\Fscr$ is disjoint from $F_0$, and let $F' \in \Fscr$ be an arbitrary
element distinct from $F, F_1, F_2$ and $F_0$.
Then the quadruple $F_1, F_2, F, F'$ does not have an intersecting triple
as every triple contains a pair of disjoint sets. But this contradicts 
the $(4,3)$-property! 
Hence~\eqref{eq:obs1} holds as claimed.

Next, we claim that

\begin{equation}\label{eq:obs2} 
\text{If $F_1, F_2, F_3 \in \Fscr$ are such that
$F_1\cap F_2\cap F_3 \neq \emptyset$ then also
$F_0 \cap F_1\cap F_2\cap F_3 \neq \emptyset$.}
\end{equation}

To see that the claim~\eqref{eq:obs2} holds,  consider an arbitrary triple 
$F_1, F_2, F_3 \in \Fscr$ such that $F_1\cap F_2\cap F_3 \neq \emptyset$. 
Let us assume $F_1\cap F_2\cap F_3 \not\subseteq F_0$ (otherwise we are done), and
fix a $q \in (F_1\cap F_2\cap F_3)\setminus F_0$.
By considering the quadruple
$A, B, F_1, F_2$ we see that we either have $A \cap F_1 \cap F_2 \neq \emptyset$ or
$B \cap F_1 \cap F_2 \neq \emptyset$.
In either case, there is a point $p_{12} \in F_0 \cap F_1 \cap F_2$.
Similarly there are points $p_{13} \in F_0 \cap F_1 \cap F_2, p_{23} \in F_0 \cap F_2 \cap F_3$.

By Radon's lemma the set $\{q, p_{12}, p_{13}, p_{23} \}$ can be partitioned into
two sets whose convex hulls intersect.
Note that we cannot have that $q \in \conv(\{p_{12},p_{13},p_{23}\})$ since
$q \not\in F_0$ and $p_{12},p_{13},p_{23} \in F_0$ and $F_0$ is convex.
Hence, up to relabelling of the indices we have either
$p_{23} \in \conv( \{q, p_{12}, p_{13}\} )$ or $[q,p_{23}] \cap [p_{12},p_{13}] \neq \emptyset$.

In the first case we have that $p_{23} \in F_0 \cap F_1 \cap F_2 \cap F_3$ 
since we have chosen $p_{23} \in F_0 \cap F_2 \cap F_3$ and
$\conv( \{q, p_{12}, p_{13}\} ) \subseteq F_1$ as all three of
$q,p_{12},p_{13} \in F_1$ and $F_1$ is convex.

In the second case we have that the intersection point of $[q,p_{23}]$ and $[p_{12},p_{13}]$
is in $F_0 \cap F_1 \cap F_2 \cap F_3$.
This is because $[q,p_{23}] \subseteq F_2 \cap F_3$ and $[p_{12},p_{13}] \subseteq F_0 \cap F_1$.

Thus,~\eqref{eq:obs2} holds as claimed.

We now define a new collection of sets by setting:

\[ \Fscr' := \{ F \cap F_0 : F \in \Fscr \}. \]

\noindent
Since the members of $\Fscr$ are closed and convex and $F_0$ is compact and convex, each element of $\Fscr'$ is compact and convex.
By~\eqref{eq:obs1} each set of $\Fscr'$ is nonempty (this is needed since otherwise there cannot
be any transversal of $\Fscr'$), and by~\eqref{eq:obs2} together with the
fact that $\Fscr$ satisfies the $(4,3)$-property, the collection $\Fscr'$ also satisfies the $(4,3)$-property.
The theorem now follows from Corollary~\ref{cor:Kleitman} as every transversal of $\Fscr'$ is also a transversal
of $\Fscr$.
\end{proofof}

\section*{Acknowledgement}

I thank Bart de Keijzer and Branko Gr\"unbaum for helpful discussions.

\bibliographystyle{plain}
\bibliography{GrunbaumConjecture}

\end{document}

%% file: Fn4.pspdftex
\begin{picture}(0,0)%
\includegraphics{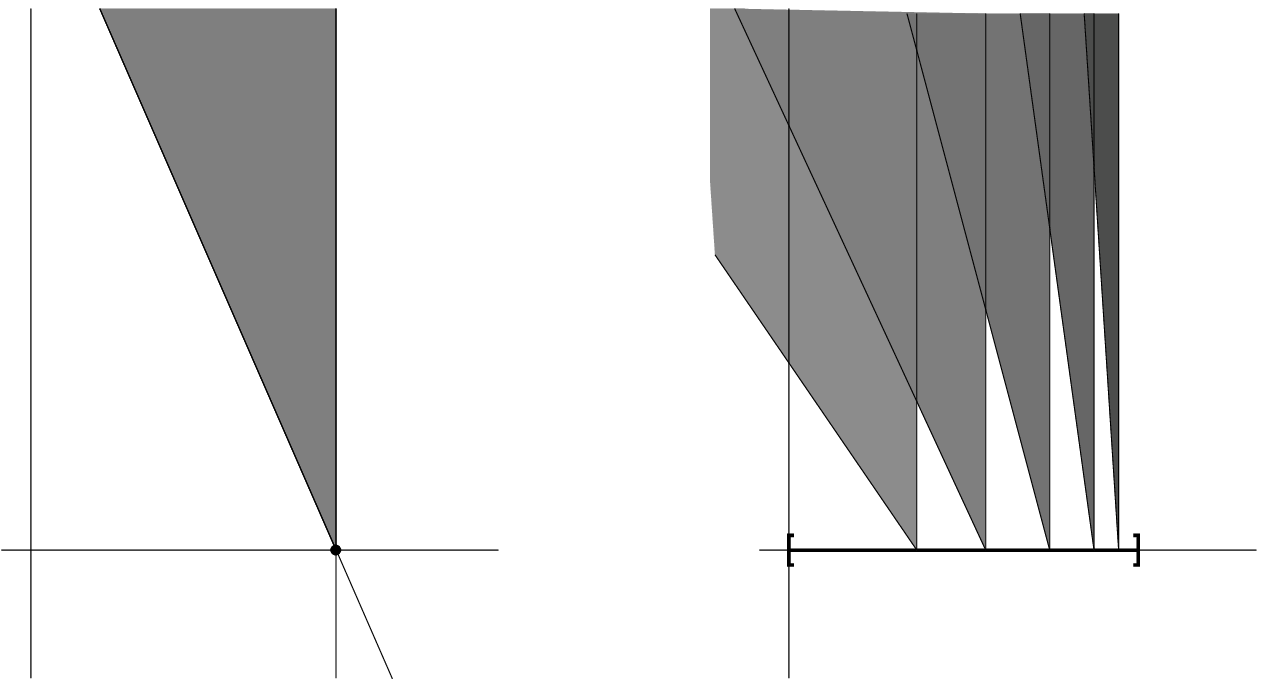}%
\end{picture}%
\setlength{\unitlength}{2072sp}%
\begingroup\makeatletter\ifx\SetFigFont\undefined%
\gdef\SetFigFont#1#2#3#4#5{%
  \reset@font\fontsize{#1}{#2pt}%
  \fontfamily{#3}\fontseries{#4}\fontshape{#5}%
  \selectfont}%
\fi\endgroup%
\begin{picture}(11499,6150)(1519,-5749)
\put(4681,-4471){\makebox(0,0)[lb]{\smash{{\SetFigFont{6}{7.2}{\rmdefault}{\mddefault}{\updefault}{\color[rgb]{0,0,0}$p_n$}%
}}}}
\put(12631,-4744){\makebox(0,0)[lb]{\smash{{\SetFigFont{6}{7.2}{\rmdefault}{\mddefault}{\updefault}{\color[rgb]{0,0,0}$x$-axis}%
}}}}
\put(8686,-3976){\rotatebox{90.0}{\makebox(0,0)[lb]{\smash{{\SetFigFont{6}{7.2}{\rmdefault}{\mddefault}{\updefault}{\color[rgb]{0,0,0}$y$-axis}%
}}}}}
\put(3835,-1262){\makebox(0,0)[lb]{\smash{{\SetFigFont{6}{7.2}{\rmdefault}{\mddefault}{\updefault}{\color[rgb]{0,0,0}$F_n$}%
}}}}
\put(5701,-4744){\makebox(0,0)[lb]{\smash{{\SetFigFont{6}{7.2}{\rmdefault}{\mddefault}{\updefault}{\color[rgb]{0,0,0}$x$-axis}%
}}}}
\put(1756,-2311){\rotatebox{90.0}{\makebox(0,0)[lb]{\smash{{\SetFigFont{6}{7.2}{\rmdefault}{\mddefault}{\updefault}{\color[rgb]{0,0,0}$y$-axis}%
}}}}}
\put(4681,-2266){\rotatebox{270.0}{\makebox(0,0)[lb]{\smash{{\SetFigFont{6}{7.2}{\rmdefault}{\mddefault}{\updefault}{\color[rgb]{0,0,0}$\ell_n$}%
}}}}}
\put(2791,-916){\rotatebox{292.0}{\makebox(0,0)[lb]{\smash{{\SetFigFont{6}{7.2}{\rmdefault}{\mddefault}{\updefault}{\color[rgb]{0,0,0}$\ell_n'$}%
}}}}}
\put(9811,-4786){\makebox(0,0)[lb]{\smash{{\SetFigFont{6}{7.2}{\rmdefault}{\mddefault}{\updefault}{\color[rgb]{0,0,0}$F_1 \cap F_2$}%
}}}}
\end{picture}%